\documentclass[secthm,seceqn,amsthm,ussrhead,12pt]{amsart}
\usepackage{amsmath,latexsym}
\usepackage[english]{babel}
\usepackage[psamsfonts]{amssymb}
\usepackage{times}
\usepackage{cite}

\usepackage[mathcal]{euscript}
\numberwithin{equation}{section} \textwidth=15.5cm
\newtheorem{theorem}{Theorem}[section]

\newtheorem{lemma}[theorem]{Lemma}

\theoremstyle{definition}
\newtheorem{remark}[theorem]{Remark}

\usepackage[colorlinks, bookmarks=true]{hyperref}
\usepackage{color,graphicx,shortvrb}

\usepackage{enumerate}
\usepackage{amsmath}
\usepackage{amssymb}
\usepackage[usenames,dvipsnames]{xcolor}

\usepackage[latin 1]{inputenc}

\def\11{\textbf{$1$}}

\begin{document}

\numberwithin{equation}{section}

\title[Local derivations on finite-dimensional Lie  algebras]{Local derivations on finite-dimensional Lie  algebras}

\author[Ayupov]{Shavkat Ayupov}
\email{sh$_{-}$ayupov@mail.ru}
\address{Dormon yoli str.,29, Institute of
 Mathematics,  National University of
Uzbekistan,  100125  Tashkent,   Uzbekistan}

\author[Kudaybergenov]{Karimbergen Kudaybergenov}
\email{karim2006@mail.ru}
\address{Ch. Abdirov str., 1, Department of Mathematics, Karakalpak State University, Nukus 230113, Uzbekistan}

\date{\today}
\maketitle

\maketitle \thispagestyle{empty}

\maketitle

\begin{abstract} We prove that every local
derivation on a finite-dimensional semisimple  Lie algebra over
an algebraically closed field of characteristic zero is a
derivation. We also give  examples of  finite-dimensional nilpotent Lie
algebras $\mathcal{L}$ with $\dim\mathcal{L}\geq 3$ which admit  local
derivations which are not  derivations.

\\


\end{abstract}

\maketitle \thispagestyle{empty}

\section{Introduction}\label{sec:intro}

 In 1990, Kadison \cite{Kad90}
and Larson and Sourour \cite{LarSou} introduce the following concept of
local derivation: let $X$ be a Banach
$A$-bimodule over a Banach algebra $A$, a linear mapping
$\Delta:A\to X$ is said to be a \textit{local derivation} if for
every $x$ in $A$ there exists a derivation $D_x :A\to X$,
depending on $x$, satisfying $\Delta(x) = D_x (x).$

The main problems concerning this notion are to find conditions
under which local derivations become derivations and to present
examples of algebras with local derivations that are not
derivations \cite{Bresar, Kad90, LarSou}.
Kadison proves
in \cite[Theorem A]{Kad90} that each continuous local derivation
of a von Neumann algebra $M$ into a dual Banach $M$-bimodule is a
derivation. This theorem gave rise to studies and several results
 on local derivations on C$^*$-algebras, culminating with a
definitive contribution due to Johnson, which asserts that every
continuous local derivation of a C$^*$-algebra $A$ into a Banach
$A$-bimodule is a derivation \cite[Theorem 5.3]{John01}. Moreover in his
 paper, Johnson also gives an automatic continuity
result by proving that local derivations of a C$^*$-algebra $A$
into a Banach $A$-bimodule $X$ are continuous even if not assumed
a priori to be so (cf. \cite[Theorem 7.5]{John01}).\smallskip

Investigation of local derivations on (non necessarily Banach) algebras
of unbounded operators were initiated in papers \cite{AAKN11} and \cite{AKNA14}.

The  paper \cite{AAKN11} is devoted to the study of local derivations on
the algebra $S(M, \tau)$ of all $\tau$-measurable operators
affiliated with a von Neumann algebra $M$ and a faithful normal
semi-finite trace $\tau.$ One of  main results in the mentioned paper
presents an unbounded version of Kadison's Theorem A from \cite{Kad90} and
it asserts that every local derivation on $S(M, \tau)$ which is
continuous in the measure topology automatically becomes a
derivation. In particular in the case of the type I von Neumann
algebra $M$ all such local derivations on $S(M, \tau)$ are inner
derivations. Moreover for type I finite
von Neumann algebras without abelian direct summands as well as
for von Neumann algebras with the atomic lattice of projections,
the continuity condition on local derivations is redundant.
In \cite{AKNA14} it was proved that each local derivation on the so-called
non commutative Arens algebras affiliated with a von Neumann algebra $M$ and a faithful normal
semi-finite trace $\tau$ is automatically a derivation.

The paper \cite{AAKN11}  also deals with the problem of existence of local
derivations which are not derivations on algebras of measurable
operators. The consideration of such examples on various finite-
and infinite dimensional algebras was initiated by Kadison,
Kaplansky and Jensen (see \cite{Kad90}).  In \cite{AAKN11} this problem has been solved for a
class of commutative regular algebras, which include the algebras
of measurable functions on a measure space. Namely necessary and sufficient
conditions were obtained for the algebras of measurable and $\tau$-measurable
operators affiliated with a commutative von Neumann algebra to
admit local derivations that are not derivations.

In \cite{AK15} we initiated the study of derivation type maps on
 non associative algebras, namely, we investigated so-called  2-local derivations
on finite-dimensional Lie algebras, and showed an essential difference between
semisimple and nilpotent Lie algebras in the behavior of their 2-local derivations.

  The present paper is devoted to local derivation
on finite-dimensional Lie algebra over an algebraically closed field of characteristic zero.

After preliminaries  we prove in Section 3 the main result of the paper which asserts that
every local derivation on a finite dimensional semisimple Lie algebra over an algebraically closed field
of zero characteristic, is automatically a derivation. In Section 4 we give  examples of nilpotent Lie
algebra (so-call filiform Lie algebras) which admit  local derivations which are not  derivations.

\section{Preliminaries}
\label{lie}

All algebras and vector spaces considered in the paper are over an
algebraically closed field $\mathbb{F}$ with zero characteristic.

Let $\mathcal{L}$ be a Lie algebra. The \textit{center} of
$\mathcal{L}$ is denoted by $Z(\mathcal{L}):$
$$
Z(\mathcal{L})=\{x\in \mathcal{L}: [x,y]=0,\,\forall\,  y\in
\mathcal{L}\}.
$$

A Lie algebra $\mathcal{L}$ is called \textit{nilpotent} (respectively \textit{solvable})
if $
\mathcal{L}^k=\{0\}$  (respectively, if $\mathcal{L}^{(k)}=\{0\}$)  for some integer $k,$  where
$\mathcal{L}^0=\mathcal{L},$ $\mathcal{L}^k=[\mathcal{L}^{k-1,},\mathcal{L}],$ (respectively,  $\mathcal{L}^{(0)}=\mathcal{L},$
$\mathcal{L}^{(k)}=[\mathcal{L}^{(k-1)}, \mathcal{L}^{(k-1)}]), \,k\geq1.$ It is clear that nilpotent Lie algebras are solvable.

 Any Lie algebra $\mathcal{L}$ contains a unique maximal
solvable ideal, called the radical of $\mathcal{L}$ and denoted
$\mbox{Rad} \mathcal{L}.$ A non trivial Lie algebra $\mathcal{L}$
is called \textit{semisimple} if $\mbox{Rad} \mathcal{L}=0.$ This condition
is equivalent to requiring that $\mathcal{L}$ have no nonzero
abelian ideals.

A \textit{derivation} on a Lie algebra $\mathcal{L}$ is a linear
map $D:\mathcal{L}\rightarrow \mathcal{L}$ which satisfies the
Leibniz rule, that is
$$
D([x,y])=[D(x),y]+[x,D(y)]
$$
for all $x,y\in \mathcal{L}.$ The set of all derivations of a Lie
algebra $\mathcal{L}$ is a Lie algebra with respect to commutation
operation and it is denoted by $\mbox{Der} \mathcal{L}.$ For any
$a\in \mathcal{L},$ the map $\mbox{ad}(a):\mathcal{L}\rightarrow \mathcal{L}$
defined as $\mbox{ad}(a)(x) = [a,x], x\in \mathcal{L},$ is a derivation, and derivations
of this form are called \textit{inner derivation}. The set of all
inner derivations of $\mathcal{L}$ denoted $\mbox{ad} L$ is an
ideal in $\mbox{Der} \mathcal{L}.$ It is well known that any
derivation on a finite-dimensional semisimple Lie algebra is
inner.

Recall that a  linear map $\Delta:\mathcal{L}\rightarrow
\mathcal{L}$
 is called a
 \textit{local derivation} if  for every $x\in \mathcal{L},$  there exists
 a derivation $D_{x}:\mathcal{L}\rightarrow \mathcal{L}$ (depending on $x$)
such that $\Delta(x)=D_{x}(x).$

\section{Local derivations of finite-dimensional  semisimple Lie algebras}\label{semisimple}

The main result of this paper is the following theorem.

\begin{theorem}\label{localsemisimple}
Let $\mathcal{L}$ be a finite-dimensional  semisimple Lie algebra. Then
any local derivation $\Delta$ on $\mathcal{L}$ is a derivation.
\end{theorem}

 A \textit{Cartan subalgebra} $\mathcal{H}$ of   a semisimple Lie algebra $\mathcal{L}$ is a nilpotent
subalgebra which coincides with its centralizer: $C(\mathcal{H}) =
\{x\in \mathcal{L}: [x, h] = 0,\, \forall h\in \mathcal{H}\} =
\mathcal{H}.$

A Cartan subalgebra $\mathcal{H}$ of a  finite-dimensional
semisimple Lie algebra $\mathcal{L}$ is \textit{abelian}, i.e.
$[x,y]=0$ for all $x,y\in \mathcal{H}.$

From now on in this section, we fix a  semisimple Lie
algebra $\mathcal{L}$ and a Cartan subalgebra $\mathcal{H}\subset
\mathcal{L}.$

We will essentially use the following decomposition of
finite-dimensional  semisimple Lie algebras (see for details
\cite{Jacob}, \cite{Humphreys}).

There exists a  decomposition for $\mathcal{L},$ called the
\textit{root decomposition}
$$
\mathcal{L}=\mathcal{H}\oplus \bigoplus\limits_{\alpha\in
R}\mathcal{L}_\alpha,
$$
where
$$
\mathcal{L}_\alpha=\{x\in \mathcal{L}:  [h, x] = \alpha(h)x,\,
\forall\,  h \in \mathcal{H}\},
$$
$$
R = \{\alpha \in \mathcal{H}^\ast\setminus \{0\}:
\mathcal{L}_\alpha\neq \{0\}\}
$$
and $\mathcal{H}^\ast$ is the space of all linear functionals on
$\mathcal{H}.$ The set $R$ is called the \textit{root system} of
$\mathcal{L},$ and subspaces $\mathcal{L}_\alpha$ are called the
\textit{root subspaces}.

The above decomposition has the following important properties: If
for each $\alpha\in R$ we take a non zero element $e_\alpha\in
\mathcal{L}_\alpha,$ then

(a) $[e_\alpha, e_\beta]=n_{\alpha,\beta}e_{\alpha+\beta},$ if
$\alpha+\beta\neq 0$ is a root, where $0\neq n_{\alpha,\beta}\in
\mathbb{F}$;

(b) $[e_\alpha, e_\beta]=0,$ if $\alpha+\beta \neq 0$ is not a
root;

(c) $[e_\alpha, e_{-\alpha}]=h_\alpha\in \mathcal{H};$

(d) $\mathcal{L}_\alpha=\mathbb{F}e_\alpha$ for all $\alpha\in R.$

From the definition of the root subspaces it follows that
\begin{center}
$[h, e_\alpha]=\alpha(h)e_\alpha$ for all $h\in \mathcal{H},$
$\alpha\in R.$
\end{center}

There exists a basis  $B = \{\alpha_1, \ldots , \alpha_l\}$ of
$\mathcal{H}^\ast$ such that any root $\alpha\in R$ is a linear
combination of the $\{\alpha_i\}_{1\leq i\leq l}$ with integer
coefficients (see \cite{Humphreys}).

From now on we fixed co-called \textit{Chevalley basis}
$\left\{h_{i}=h_{\alpha_i}: i\in \overline{1,l}\right\}
\cup\left\{e_\alpha: \alpha\in R\right\}$ in $\mathcal{L}$ with
the property that all structure constants are integers, in
particular, $\alpha(h_i)$ is integer for all $\alpha\in R,$ $i\in
\overline{1,l}$ and $n_{\alpha,\beta}$ is also  integer for all
$\alpha, \beta\in R$ with $\alpha+\beta\in R.$

Recall that a Lie algebra is \textit{simple} if  it has no non-trivial ideals and is not abelian.
Any semisimple Lie algebra is the direct sum of its minimal ideals, which are canonically determined simple Lie algebras.

 The following algebras are simple finite-dimensional Lie algebras:

\begin{itemize}
\item  $A_n:$ $\mathfrak {sl}_{n+1}(\mathbb{F}),$ the special linear Lie algebra;
\item $B_n:$ $\mathfrak{so}_{2n+1}(\mathbb{F}),$ the odd-dimensional special
orthogonal Lie algebra;
\item $C_n:$ $\mathfrak {sp}_{2n}(\mathbb{F}),$ the
symplectic Lie algebra;
\item $D_n:$ $\mathfrak{so}_{2n}(\mathbb{F}),$ the
even-dimensional special orthogonal Lie algebra.
\end{itemize}
These Lie algebras are numbered so that $n$ is the rank, i.e. the
dimension of Cartan subalgebra.  These four families, together
with five exceptions ($\mathfrak{e}_6,$ $\mathfrak{e}_7,$
$\mathfrak{e}_8,$ $\mathfrak{f}_4$ and $\mathfrak{g}_2$), are in
fact the only simple Lie algebras over  an algebraically closed
field of characteristic zero (see \cite{Humphreys}).

The proof  of Theorem~\ref{localsemisimple} consists of three
steps. In the first  step we will  show that any local derivation
$\Delta$ on semisimple Lie algebra can be represent in the form
$$
\Delta=T+\textrm{ad}(a),
$$
where $T$ is a local derivation such that $T|_{\mathcal{H}}\equiv
0$ and $a\in \mathcal{L}.$

Let us rewrite a root decomposition of  $\mathcal{L}$ as
$$
\mathcal{L}=\mathcal{L}_1\oplus \mathcal{L}_2,
$$
where $\mathcal{L}_1=H,$ $\mathcal{L}_2=\textrm{span}\{e_\alpha :
\alpha\in R\}.$  Then any local derivation on $\mathcal{L}$ can be
represent as  $2\times 2$-matrix of the following form:
$$
\left(%
\begin{array}{cc}
  A_{11} & A_{12} \\
  A_{21} & A_{22} \\
\end{array}%
\right),
$$
where $A_{ij}$ maps $\mathcal{L}_j$ into $\mathcal{L}_i$ for all
$1\leq i,j\leq 2.$

Let $\{h_1, \cdots, h_l\}$ be a basis of $\mathcal{H}.$
For
$$
x=\sum\limits_{i=1}^l \lambda_ih_i+\sum\limits_{\alpha\in R}
\lambda_\alpha e_\alpha
$$
we denote
\begin{center}
$x_i=\lambda_i$ and $x_\alpha=\lambda_\alpha$
\end{center}
for all $i\in \overline{1, l},$ $\alpha\in R.$

For any $h_k$ $(1\leq k \leq l)$ take an element
$a=h_a+\sum\limits_{\gamma\in R}a_\gamma^{(k)}e_{\gamma}$
(depending on $h_k$) such that $\Delta(h_k)=[a, h_k].$ Since
$$
\Delta(h_k)=[a, h_k]=\left[h_a+\sum\limits_{\gamma\in
R}a_\gamma^{(k)}e_{\gamma}, h_k\right]=\sum\limits_{\gamma\in R}
a_\gamma^{(k)} \gamma(h_k)e_{\gamma},
$$
it follows that $\Delta(h_k)_i=0$ for all $i\in \overline{1,l}.$
This means that  $A_{11}=0.$ We also see that
$A_{21}=\left(a_{\gamma,k}\right),$ where
$a_{\gamma,k}=a_{\gamma}^{(k)}\gamma(h_k)$  for all $\gamma \in R,
k\in \overline{1, l}.$ So, the matrix of $\Delta$ has the
following form
\begin{equation}\label{matrix}
\left(%
\begin{array}{cc}
  0      & A_{12} \\
  A_{21} & A_{22} \\
\end{array}%
\right).
\end{equation}

\begin{lemma}\label{jordanzero}
For every $\gamma\in R$ there exist  $d_\gamma\in \mathbb{F}$ and
integers  $r_{\gamma, i},$ $i=1,\cdots, l$ such that $a_{\gamma,
i}=r_{\gamma, i} d_\gamma$ for all $i.$
\end{lemma}

\begin{proof} Let $\gamma_0\in R$ be a fixed root. Since $\{h_1, \cdots, h_l\}$ is a basis of $\mathcal{H},$
there exists $k\in \overline{1, l}$ such that $\gamma_0(h_k)\neq
0.$ Set
$$
d_{\gamma_0}=\frac{a_{\gamma_0,k}}{\gamma_0(h_k)}.
$$
For $i\neq k$ put $h=\gamma_0(h_k)h_i-\gamma_0(h_i)h_k.$ Using
\eqref{matrix} we obtain that
\begin{eqnarray*}
\Delta(h)_{\gamma_0} & = & a_{\gamma_0,
i}\gamma_0(h_k)-a_{\gamma_0, k}\gamma_0(h_i)=a_{\gamma_0,
i}\gamma_0(h_k) - d_{\gamma_0} \gamma_0(h_k) \gamma_0(h_i).
\end{eqnarray*}
On the other hand, we can find an element $b$ (depending on $h_k$)
such that $\Delta(h_k)=[b, h_k].$ Then
\begin{eqnarray*}
\Delta(h)_{\gamma_0} & = & [b,
h]_{\gamma_0}=\left[h_b+\sum\limits_{\gamma\in R}b_\gamma
e_{\gamma},
\gamma_0(h_k)h_i-\gamma_0(h_i)h_k\right]_{\gamma_0}=\\
& = & -b_{\gamma_0} \gamma_0(h_k)\gamma_0(h_i)+b_{\gamma_0}
\gamma_0(h_i)\gamma_0(h_k)=0.
\end{eqnarray*}
Therefore
$$
a_{\gamma_0,i}=\gamma_0(h_i) d_{\gamma_0}
$$
for all $i=1,\cdots, l.$ Since we use Shevalley basis, it follows
that all $\gamma_0(h_i)$ are integers. The proof is complete.
\end{proof}

Set
\begin{equation}\label{transc}
h_0=\sum\limits_{k=1}^lt^k h_k,
\end{equation}
where $t$ is a fixed algebraic number  from $\mathbb{F}$ of a
degree bigger than $l=\dim\mathcal{H}.$

\begin{lemma}\label{zeroo}
Suppose that $\Delta(h_0)=0.$ Then $\Delta|_H\equiv 0.$
\end{lemma}

\begin{proof}
It is suffices to show that $A_{21}=0$ for the matrix of the local
derivation $\Delta.$

For any  $\gamma\in R$ we have
\begin{eqnarray*}
0 & = & \Delta(h_0)_\gamma = \sum\limits_{k=1}^l a_{\gamma, k} t^k
= \sum\limits_{k=1}^l r_{\gamma, k} d_\gamma t^k = d_\gamma
\sum\limits_{k=1}^l r_{\gamma, k} t^k.
\end{eqnarray*}
Since  $r_{\gamma, 1}, \cdots,  r_{\gamma, l}$ are integers and
the degree of the algebraic number  $t$ is bigger than $l,$ it
follows that $\sum\limits_{k=1}^l r_{\gamma, k} t^k\neq 0.$
Therefore  $d_\gamma=0,$ i.e. $a_{\gamma, k}=0$ for all $\gamma$
and $k.$ This means that $A_{21}=0.$ The proof is complete.
\end{proof}

Let us take an element $a$ such that $\Delta(h_0)=[a, h_0].$ Set
\begin{equation}\label{dec}
T=\Delta-\textrm{ad}(a).
\end{equation}
By Lemma~\ref{zeroo} we have that $T|_{\mathcal{H}}\equiv 0.$

Now we are in position to pass to the second   step of our proof.
In this step we show that if a local derivation $\Delta$
annihilates  a Cartan subalgebra $\mathcal{H}$, then it leaves  invariant
the root subspaces.

Let us   first  consider local derivations on a Lie algebra which
is a direct sum of algebras $\mathfrak{sl}_{n_i+1}(\mathbb{F}),$
where $n_1, \cdots, n_k\in \mathbb{N}.$

\begin{lemma}\label{typea}
Any local derivation $\Delta$ on $\mathfrak{sl}_{n+1}(\mathbb{F})$
is a derivation.
\end{lemma}

\begin{proof}
Let $\Delta$ be a local derivation on
$\mathfrak{sl}_{n+1}(\mathbb{F}).$ Let us extend $\Delta$ on
$\mathfrak{gl}_{n+1}(\mathbb{F})$ by
$$
\Delta_0(x)=\Delta(x-\textrm{tr}(x)\mathbf{1}),\, x\in
\mathfrak{gl}_{n+1}(\mathbb{F}),
$$
where $\mathbf{1}$ is the identity matrix in
$\mathfrak{gl}_{n+1}(\mathbb{F})$ and $\textrm{tr}$ is the
 trace on $\mathfrak{gl}_{n+1}(\mathbb{F})$ with $\textrm{tr}(\mathbf{1})=1.$  Since
$$
\Delta_0(x)=\Delta(x-\textrm{tr}(x)\mathbf{1})=[a_{x-\textrm{tr}(x)\mathbf{1}},
x-\textrm{tr}(x)\mathbf{1}]=[a_{x-\textrm{tr}(x)\mathbf{1}}, x],
$$
it follows that $\Delta_0$ is a local associative derivation on
$M_{n+1}(\mathbb{F}).$ By \cite[Theorem 2.3]{Bresar} we have that
$\Delta_0$ is an associative derivation and therefore a Lie
derivation. Hence $\Delta$ is also a derivation. The proof is complete.
\end{proof}

Suppose  that $\mathcal{L}=\bigoplus_{i=1}^m \mathcal{L}_i$ is a
direct sum of  semisimple Lie algebras $\mathcal{L}_i.$ Since the
center of semisimple Lie algebra  $\mathcal{L}$ is trivial,
\cite[P. 9, Theorem 1]{Goze} implies  that
$$
\textrm{Der}(\mathcal{L})=\bigoplus_{i=1}^m
\textrm{Der}(\mathcal{L}_i).
$$
Thus any local derivation $\Delta$ on $\mathcal{L}$ can be
decomposed as
$$
\Delta=\Delta_1+\cdots +\Delta_m,
$$
where $\Delta_i$ is a local derivation on $\mathcal{L}_i$ for all
$i\in \overline{1,m}.$

This property together with Lemma~\ref{typea} imply the following
result.

\begin{lemma}\label{typeasum}
Any local derivation $\Delta$ on $\bigoplus\limits_{i=1}^m
\mathfrak{sl}_{n_i+1}(\mathbb{F})$ is a derivation.
\end{lemma}

Let $\alpha, \beta\in R.$ There exist integers $p$ and $q$ such
that all
$$
\beta-p\alpha, \cdots, \beta, \cdots, \beta+q\alpha
$$
are roots, and this finite sequence is said to be
$\alpha$-\textit{string through} $\beta.$ The $\alpha$-string
through $\beta$ contains at most four roots (see \cite{Jacob}).

\begin{lemma}\label{restriction}
Let $\Delta$ be a local derivation such that $\Delta|_H\equiv 0.$
Then $\Delta$ maps $\textrm{span}\left\{h_\alpha, e_\alpha,
e_{-\alpha}\right\}$ into itself for all $\alpha\in R.$ Moreover
\begin{equation}\label{calpha}
\Delta(e_{\pm\alpha})=\pm c_\alpha e_{\pm\alpha}
\end{equation}
for some $c_\alpha\in \mathbb{F}.$
\end{lemma}

\begin{proof} Let $\alpha\in R.$

We are going to show that $\Delta(e_{\alpha})_\gamma=0$ for all
$\gamma\in R$ with $\gamma\neq\pm \alpha.$

Take an element $a=h_a+\sum\limits_{\gamma\in R}a_\gamma
e_{\gamma}$ (depending on $e_{\alpha}$) such that
$\Delta(e_{\alpha})=[a, e_{\alpha}].$ Then
\begin{eqnarray}\label{alphazero}
\Delta(e_{\alpha}) & = & \alpha(h_a)e_{\alpha}-a_{-\alpha}
h_\alpha +\sum\limits_{\gamma+\alpha\in R} a_{\gamma}
n_{\gamma,\alpha} e_{\gamma+\alpha}.
\end{eqnarray}

This equality implies that $\Delta(e_{\alpha})_{\gamma}=0$ if
$\gamma\neq\pm \alpha$ and $\gamma-\alpha$ is not a root, because
$\gamma$ can not be represented  as a sum $\gamma=\gamma'+\alpha.$

Now let  $\beta$ be a root such that $\beta+\alpha$ is also a
root.

It is suffices to consider the following three possible  cases.

Case 1. The $\alpha$-string through $\beta$ contains $2$ roots.
Without loss of generality we can assume that $\beta,
\beta+\alpha\in R$ and  $\beta-\alpha\notin R.$ Then the
equality~\eqref{alphazero} implies that
\begin{equation}\label{gammaal}
\Delta(e_{\alpha})_{\beta}=0.
\end{equation}
Since $\alpha\neq 0,$ there exists an element $h\in H$ such that
\begin{center}
$\beta(h)=1$ and $(\beta+\alpha)(h)=0.$
\end{center}
Take an element $b\in \mathcal{L}$ such that
$$
\Delta(h+e_{\alpha})=[b, h+e_{\alpha}].
$$
Taking into account that $\beta(h)=1$ and $(\beta+\alpha)(h)=0,$
we obtain
\begin{equation}\label{gammazeroa}
\Delta(h+e_{\alpha})_{\beta}= -b_{\beta},
\end{equation}
\begin{equation}\label{gammazero}
\Delta(h+e_{\alpha})_{\beta+\alpha}= b_{\beta}n_{\beta,\alpha}.
\end{equation}
Recall that  $\Delta(h+e_{\alpha})=\Delta(e_{\alpha}).$ Then
\eqref{gammaal} combined with \eqref{gammazeroa} gives us
$b_{\beta}=0.$ Thus from \eqref{gammazero} it follows that
$$
\Delta(h+e_{\alpha})_{\beta+\alpha}= 0.
$$
Again the equality  $\Delta(h+e_{\alpha})=\Delta(e_{\alpha}),$
implies that
$$
\Delta(e_{\alpha})_{\beta+\alpha}= 0.
$$

Case 2. The $\alpha$-string through $\beta$ contains $3$ roots
$\beta, \beta+\alpha, \beta+2\alpha\in R.$ As in the previous case
we obtain that
\begin{equation}\label{gam}
\Delta(e_{\alpha})_{\beta}=\Delta(e_{\alpha})_{\beta+\alpha}=0.
\end{equation}
Now take an element $h\in H$ such that
\begin{center}
$(\beta+\alpha)(h)=1$ and $(\beta+2\alpha)(h)=0.$
\end{center}
Let $c$ be an element such that
$$
\Delta(h+e_{\alpha})=[c, h+e_{\alpha}].
$$
Using the definition of the element $h$ from the last equality we
obtain
\begin{equation}\label{ga}
\Delta(h+e_{\alpha})_{\beta}= -2c_{\beta},
\end{equation}
\begin{equation}\label{gamz}
\Delta(h+e_{\alpha})_{\beta+\alpha}= c_{\beta+\alpha}+c_{\beta}
n_{\beta, \alpha},
\end{equation}
\begin{equation}\label{gamze}
\Delta(h+e_{\alpha})_{\beta+2\alpha}=
c_{\beta+\alpha}n_{\beta+\alpha,\alpha}.
\end{equation}
Combining  \eqref{gam}, \eqref{ga}, \eqref{gamz} and
\eqref{gamze}, we get $c_{\beta+\alpha}=0$ and $
\Delta(h+e_{\alpha})_{\beta+2\alpha}= 0.$ Since
$\Delta(h+e_{\alpha})=\Delta(e_{\alpha}),$ it follows that
$
\Delta(e_{\alpha})_{\beta+2\alpha}= 0.
$

Case 3. $\beta, \beta+\alpha, \beta+2\alpha, \beta+3\alpha\in R.$
The proof is similar to the previous cases.

So, $\Delta(e_{\pm\alpha})_\gamma=0$ for all $\gamma\in R$ with
$\gamma\neq\pm \alpha.$ This means that $\Delta(e_{\pm\alpha})\in
\textrm{span}\left\{h_{\alpha}, e_{\alpha}, e_{-\alpha}\right\}.$
Thus $\Delta|_{\textrm{span}\left\{h_{\alpha}, e_{\alpha},
e_{-\alpha}\right\}}$ is a local derivation on
$\textrm{span}\left\{h_{\alpha}, e_{\alpha},
e_{-\alpha}\right\}\cong \mathfrak{sl}_2(\mathbb{F}),$ and
therefore Lemma~~\ref{typea} implies that
$\Delta|_{\textrm{span}\left\{h_{\alpha}, e_{\alpha},
e_{-\alpha}\right\}}$ is a  derivation. Since
$\Delta(h_{\alpha})=0,$ there exists $c\in \mathbb{F}$ such that
$\Delta|_{\textrm{span}\left\{h_{\alpha}, e_{\alpha},
e_{-\alpha}\right\}}=\textrm{ad}(ch_\alpha).$ Therefore
$$
\Delta(e_{\pm\alpha})=[ch_{\alpha}, e_{\pm\alpha}]=\pm
c\alpha(h_\alpha)e_{\pm\alpha}.
$$
The proof is complete.
\end{proof}

In the third  step  we consider local derivations on the algebras
$\mathfrak{so}_5(\mathbb{F})$ and~$\mathfrak{g}_2.$

Let $a$ be an element from \eqref{dec}. In the proofs of the following
two Lemmas replacing, if necessary, the local derivation $\Delta$  by
$\Delta-\textrm{ad}(a),$ we may assume that
$\Delta|_{\mathcal{H}}\equiv 0.$

\subsection{Local derivations on
$\mathfrak{so}_5(\mathbb{F})$}

\begin{lemma}\label{sofive}
Any local derivation on $\mathfrak{so}_5(\mathbb{F})$ is a
derivation.
\end{lemma}

\begin{proof} Let  $\{\pm\alpha, \pm \beta, \pm(\alpha+\beta),
\pm(\alpha+2\beta)\}$ be the root system of
$\mathfrak{so}_5(\mathbb{F}).$

Since $\alpha+3\beta=(\alpha+2\beta)+\beta$ is not  a root,  it
follows that $[e_\alpha, e_{\alpha+2\beta}]=0.$ Thus
$$
\textrm{alg}\left\{e_{\pm\alpha},
e_{\pm(\alpha+2\beta)}\right\}=\textrm{span}\left\{h_\alpha,
h_{\alpha+2\beta}, e_{\pm\alpha},
e_{\pm(\alpha+2\beta)}\right\}\cong
\mathfrak{sl}_2(\mathbb{F})\oplus \mathfrak{sl}_2(\mathbb{F}),
$$
where $\textrm{alg}(S)$ is the Lie subalgebra of
$\mathfrak{so}_5(\mathbb{F})$ generated by a subset $S\subset
\mathfrak{so}_5(\mathbb{F}).$ By Lemma~\ref{restriction} $\Delta$
maps $\textrm{alg}\left\{e_{\pm\alpha},
e_{\pm(\alpha+2\beta)}\right\}$ into itself. Since
$$
\textrm{alg}\left\{e_{\pm\alpha},
e_{\pm(\alpha+2\beta)}\right\}\cong
\mathfrak{sl}_2(\mathbb{F})\oplus \mathfrak{sl}_2(\mathbb{F}),
$$
Lemma~\ref{typeasum} implies that
$\Delta|_{\textrm{alg}\left\{e_{\pm\alpha},
e_{\pm(\alpha+2\beta)}\right\}}$ is a derivation. Then there
exists an element $h_1\in H$ such that
$\Delta|_{\textrm{alg}\left\{e_{\pm\alpha},
e_{\pm(\alpha+2\beta)}\right\}}=\textrm{ad}(h_1),$ because
$\Delta|_{\textrm{alg}\left\{e_{\pm\alpha},
e_{\pm(\alpha+2\beta)}\right\}}$ is identically zero on a Cartan
subalgebra of $\textrm{alg}\left\{e_{\pm\alpha},
e_{\pm(\alpha+2\beta)}\right\}.$

Set
$$
T=\Delta-\textrm{ad}(h_1).
$$
Then
$$
T(e_{\pm\alpha})=T(e_{\pm(\alpha+2\beta)})=0.
$$

Let us  show that
$$
T(e_{\pm\beta})=T(e_{\pm(\alpha+\beta)})=0.
$$

We have
\begin{equation}\label{abe}
T(e_{\alpha}+e_{\beta}+e_{-(\alpha+2\beta)})=\mu_\beta e_\beta.
\end{equation}

On the other hand
\begin{eqnarray*}
T(e_{\alpha}+e_{\beta}+e_{-(\alpha+2\beta)}) & = & [b,
e_{\alpha}+e_{\beta}+e_{-(\alpha+2\beta)}]=\\
& = &
\alpha(h_b)e_\alpha+\beta(h_b)e_\beta-(\alpha+2\beta)(h_b)e_{-(\alpha+2\beta)}+\\
& + & \ast
e_{\beta+\alpha}+\ast h_\alpha+\ast e_{-\beta}+\\
& + & \ast e_{\beta+\alpha}+\ast e_{\alpha+2\beta}+\ast
h_{\beta}+\ast e_{-\alpha}+ \ast e_{-(\alpha+\beta)}+
\\
& + & \ast e_{-(\alpha+\beta)}+\ast e_{-\beta}+\ast
h_{\alpha+2\beta},
\end{eqnarray*}
where the symbols $\ast$ denote appropriate coefficients. We see that  the last three
rows in this equality does not contain $e_\alpha, e_\beta$ and
$e_{-(\alpha+2\beta)}.$ Comparing  the last  equality with
\eqref{abe}, we obtain that
$$
\alpha(h_b)=(\alpha+2\beta)(h_b)=0
$$
and
$$
\mu_\beta=\beta(h_b).
$$
The first two equalities give us $\alpha(h_b)=\beta(h_b)=0,$ and
therefore $\mu_\beta=0.$ Thus
$$
T(e_{\beta})=0.
$$

In a similar way we obtain
$$
T(e_{-\beta})=0.
$$

Now we will check that
$$
T(e_{\pm(\alpha+\beta)})=0.
$$

We have
$$
T(e_{-\alpha}+e_{\alpha+\beta}+e_{-(\alpha+2\beta)})=\mu_{\alpha+\beta}
e_{\alpha+\beta}.
$$

On the other hand
\begin{eqnarray*}
T(e_{-\alpha}+e_{\alpha+\beta}+e_{-(\alpha+2\beta)}) & = & [c,
e_{-\alpha}+e_{\alpha+\beta}+e_{-(\alpha+2\beta)}]=\\
& = &
-\alpha(h_c)e_{-\alpha}+(\alpha+\beta)(h_c)e_{\alpha+\beta}-(\alpha+2\beta)(h_c)e_{-(\alpha+2\beta)}+\\
& + & \ast
h_{\alpha}+\ast e_{\beta}+\ast e_{-\alpha-\beta}+\\
& + & \ast e_{\alpha+2\beta}+\ast e_{\beta}+\ast e_{\alpha}+\ast
h_{\alpha+\beta}+ \ast e_{-\beta}+
\\
& + & \ast e_{-(\alpha+\beta)}+\ast e_{-\beta}+\ast
h_{\alpha+2\beta}.
\end{eqnarray*}
As in the previous case comparing coefficients of $e_{_\alpha},
e_{\alpha+\beta}$ and $e_{-(\alpha+2\beta)}$ in the last two
equalities, we obtain that
$$
\alpha(h_c)=(\alpha+2\beta)(h_c)=0
$$
and
$$
\mu_{\alpha+\beta}=(\alpha+\beta)(h_c).
$$
The first  equalities give us $\alpha(h_c)=\beta(h_c)=0,$ and
therefore $\mu_{\alpha+\beta}=0.$ Thus
$$
T(e_{\alpha+\beta})=0.
$$

Similarly
$$
T(e_{-(\alpha+\beta)})=0.
$$
So, we have proved that $T=0.$ Thus $\Delta=\textrm{ad}(h_1)$ is a
derivation. The proof is complete.
\end{proof}

\subsection{Local derivations on the exceptional Lie algebra
$\mathfrak{g}_2$}

\begin{lemma}\label{gtwo}
Any local derivation on $\mathfrak{g}_2$ is a derivation.
\end{lemma}

\begin{proof}
Let  $\{\pm\alpha, \pm \beta, \pm(\alpha+\beta),
\pm(2\alpha+\beta), \pm(3\alpha+\beta), \pm (3\alpha+2\beta)\}$ be
the root system of $\mathfrak{g}_2.$

Since $[e_\alpha, e_{3\alpha+2\beta}]=0,$ it follows  that
$$
\textrm{alg}\left\{e_{\pm\alpha},
e_{\pm(3\alpha+2\beta)}\right\}=\textrm{span}\left\{h_\alpha,
h_{3\alpha+2\beta}, e_{\pm\alpha},
e_{\pm(3\alpha+2\beta)}\right\}\cong
\mathfrak{sl}_2(\mathbb{F})\oplus \mathfrak{sl}_2(\mathbb{F}).
$$
By Lemma~\ref{restriction} the local derivation $\Delta$ maps
$\textrm{alg}\left\{e_{\pm\alpha},
e_{\pm(3\alpha+2\beta)}\right\}$ into itself. Since
$$
\textrm{alg}\left\{e_{\pm\alpha},
e_{\pm(3\alpha+2\beta)}\right\}\cong
\mathfrak{sl}_2(\mathbb{F})\oplus \mathfrak{sl}_2(\mathbb{F})
$$
Lemma~\ref{typeasum} implies that
$\Delta|_{\textrm{alg}\left\{e_{\pm\alpha},
e_{\pm(3\alpha+2\beta)}\right\}}$ is a derivation. Therefore there
exists $h_1\in H$ such that
$\Delta|_{\textrm{alg}\left\{e_{\pm\alpha},
e_{\pm(3\alpha+2\beta)}\right\}}=\textrm{ad}(h_1).$

Set
$$
T=\Delta-\textrm{ad}(h_1).
$$
Then
$$
T(e_{\pm\alpha})=T(e_{\pm(3\alpha+2\beta)})=0.
$$

Let us to show that
$$
T(e_{\sigma})=0,
$$
where $\sigma=\pm \beta, \pm(\alpha+\beta), \pm(2\alpha+\beta),
\pm(3\alpha+\beta).$

Let us consider the case   $\sigma=\pm\beta.$ Then
\begin{equation}\label{alphabeta}
T(e_{\alpha}+e_{\beta}+e_{-(3\alpha+2\beta)})=\mu_\beta e_\beta.
\end{equation}

On the other hand
\begin{eqnarray*}\label{alpha}
T(e_{\alpha}+e_{\beta}+e_{-(3\alpha+2\beta)}) & = & [c,
e_{\alpha}+e_{\beta}+e_{-(3\alpha+2\beta)}]=\\
& = &
\alpha(h_c)e_\alpha+\beta(h_c)e_\beta-(3\alpha+2\beta)(h_c)e_{-(3\alpha+2\beta)}+\\
& + & \sum\limits_{\gamma\in R} \left(c_\gamma n_{\gamma, \alpha}
e_{\gamma+\alpha}+c_{\gamma}n_{\gamma,
\beta}e_{\gamma+\beta}+c_{\gamma} n_{\gamma,
-(3\alpha+2\beta)}e_{\gamma-3\alpha-\beta}\right).
\end{eqnarray*}
Direct computations shows that the third summand in the last
equality  does not contain  $e_\alpha, e_\beta$ and
$e_{-(3\alpha+2\beta)}.$ Comparing these components with same
components in \eqref{alphabeta}, we obtain that
$$
\alpha(h_c)=(3\alpha+2\beta)(h_c)=0
$$
and
$$
\mu_\beta=\beta(h_c).
$$
The first  equalities give us $\alpha(h_c)=\beta(h_c)=0,$ and
therefore $\mu_\beta=0.$ Thus $T(e_{\beta})=0.$

In the same way we obtain that  $T(e_{-\beta})=0.$

The remaining  cases  of $\sigma$ are similar and can be checked in the
following order:
$$
\pm(2\alpha+\beta),  \pm(3\alpha+\beta), \pm(\alpha+\beta).
$$
In these cases instead of
$e_{\alpha}+e_{\beta}+e_{-(3\alpha+2\beta)}$ we should use the
following elements:
$$
e_{\pm\beta}+ e_{\pm(2\alpha+\beta)}+ e_{\pm(3\alpha+2\beta)},
\,\,\, e_{\pm\beta}+ e_{\pm(3\alpha+\beta)}+
e_{\mp(3\alpha+2\beta)}, \,\,\, e_{\pm(3\alpha+\beta)}+
e_{\pm(\alpha+\beta)}+ e_{\mp\beta},
$$
respectively. The proof is complete.
\end{proof}

Now we are in position to give the proof  of the main result.

The main ingredient  of the proof is  reduction of the general case
to  the case of rank 2 simple Lie algebras.

\begin{proof}[\textit{Proof of Theorem~$\ref{localsemisimple}$}]
Let $\Delta:\mathcal{L}\rightarrow \mathcal{L}$ be a local
derivation and suppose that $h_0\in \mathcal{H}$ is the element defined
by~\eqref{transc}. Take an element $a\in \mathcal{L}$ (depending
on $h_0$) such that
$$
\Delta(h_0)=[a, h_0].
$$
Replacing, if necessary, the local derivation $\Delta$ by  $\Delta-\mbox{ad}(a),$ we
may assume that $\Delta(h_0)=0,$ and therefore by
Lemma~\ref{zeroo} $\Delta|_{\mathcal{H}}\equiv 0.$

Firstly note that $\Delta|_\mathcal{H}$ is a derivation, because
it is identically zero.

Let us show that
\begin{equation}\label{he}
 \Delta([h, e_\alpha])=[\Delta(h), e_\alpha]+[h, \Delta(e_\alpha)]
\end{equation}
for all $h\in \mathcal{H}$ and $\alpha\in R.$ Indeed, taking into
account that $\Delta|_\mathcal{H}\equiv 0$ and the
equality~\eqref{calpha} we have
\begin{eqnarray*}
\Delta([h, e_\alpha]) & = & \Delta(\alpha(h)e_\alpha)=\alpha(h)
\Delta(e_\alpha)=\\
& = &  \alpha(h) c_\alpha e_\alpha=c_\alpha [h, e_\alpha]=[h,
c_\alpha e_\alpha]=\\
& = & [h, \Delta(e_\alpha)]=[\Delta(h),e_\alpha]+[h,
\Delta(e_\alpha)].
\end{eqnarray*}

Now we show that
\begin{equation}\label{bss}
\vec{}\Delta([e_\alpha, e_\beta])=[\Delta(e_\alpha),
e_\beta]+[e_\alpha, \Delta(e_\beta)]
\end{equation}
for all $\alpha, \beta\in R.$ By Lemma~\ref{restriction} $\Delta$
maps $\textrm{alg}\left\{e_{\pm\alpha}, e_{\pm\beta)}\right\}$
into itself.

It is suffices to consider the following  four possible cases (see
\cite[P. 44]{Humphreys}):
\begin{enumerate}
\item[1.] $\textrm{alg}\left\{e_{\pm\alpha},
e_{\pm\beta)}\right\}\cong \mathfrak{sl}_2(\mathbb{F})\oplus
\mathfrak{sl}_2(\mathbb{F});$
\item[2.] $\textrm{alg}\left\{e_{\pm\alpha},
e_{\pm\beta)}\right\}\cong \mathfrak{sl}_3(\mathbb{F});$
\item[3.] $\textrm{alg}\left\{e_{\pm\alpha},
e_{\pm\beta)}\right\}\cong \mathfrak{so}_5(\mathbb{F});$
\item[4.] $\textrm{alg}\left\{e_{\pm\alpha},
e_{\pm\beta)}\right\}\cong \mathfrak{g}_2.$
\end{enumerate}
In the first two cases Lemma~\ref{typeasum} implies that
$\Delta|_{\textrm{alg}\left\{e_{\pm\alpha},
e_{\pm\beta)}\right\}}$ is a derivation and therefore $\Delta$
satisfies \eqref{bss}.

In a similar way in the remaining  two cases we can apply
Lemmas~\ref{sofive} and~\ref{gtwo}, respectively.

Finally, taking into account the linearity of $\Delta,$ the
equalities~\eqref{he} and~\eqref{bss}, we obtain
$$
\Delta([x,y])=[\Delta(x),y]+[x,\Delta(y)]
$$
for all $x,y\in \mathcal{L}.$ The proof is complete.
\end{proof}

\section{Local derivations on filiform   Lie algebras}\label{filiform}

In this section we consider a special class of nilpotent Lie algebras, so-called filiform Lie algebras,
and show that they  admit local derivations which are not derivations.

A Lie algebra $\mathcal{L}$ is called \textit{filiform} if $\dim
\mathcal{L}^k=n-k-1$ for $1\leq k \leq n-1,$  where
$\mathcal{L}^0=\mathcal{L},$ $\mathcal{L}^k=[\mathcal{L}^{k-1},
\mathcal{L}],\, k\geq1.$

\begin{theorem}\label{filiform}
Let $\mathcal{L}$ be a finite-dimensional filiform   Lie algebra
with $\dim \mathcal{L}\geq 3.$ Then $\mathcal{L}$ admits a local
derivation which is not a derivation.
\end{theorem}

It is well-known \cite[P. 58]{Goze}  that there is a unique
three-dimensional filiform Lie algebra $\mathcal{L}$ with a basis
$\{e_1, e_2, e_3\}$ and multiplication rule $[e_1, e_2]=e_3.$ This
Lie algebra is known as the \textit{Heisenberg algebra}.

Let $\mathcal{L}$ be a $n$-dimensional filiform Lie algebra with
$n\geq 4.$

It is known \cite{Vergne} that there exists a basis $\{e_1, e_2,
\cdots,  e_n\}$ of  $\mathcal{L}$ such that
\begin{equation}\label{basicfili}
[e_1, e_i]=e_{i+1}
\end{equation}
for all $i\in \overline{2, n-1}.$

Note that a filiform Lie algebra $\mathcal{L}$ besides
\eqref{basicfili}  may have also other non-trivial commutators.

From \eqref{basicfili} it follows that  $\{e_{k+2}, \cdots, e_n\}$
is a basis in $\mathcal{L}^k$ for all $1\leq k \leq n-2$ and
$
e_n\in Z(\mathcal{L}).
$

Since $[\mathcal{L}^1, \mathcal{L}^{n-3}]\subseteq
\mathcal{L}^{n-1}=\{0\},$ it follows that
\begin{equation}\label{nminus}
[e_i, e_{n-1}]=0
\end{equation}
for all $i=3, \cdots, n.$

Now let us define a linear operator $D$ on $\mathcal{L}$ by
\begin{equation}\label{filiformder}
D\left(\sum\limits_{k=1}^n x_ke_k\right)=\alpha x_2 e_{n-1}+\beta
x_3 e_n,
\end{equation}
where $\alpha, \beta\in \mathbb{F}.$

\begin{lemma}\label{filider}
A linear operator $D$ on $\mathcal{L}$ defined
by~\eqref{filiformder} is a derivation if and only if
$\alpha=\beta.$
\end{lemma}

\begin{proof}   Suppose that  a linear operator  $D$ defined by~\eqref{filiformder} is a  derivation.

Since $[e_1, e_2]=e_3,$ we have that
\begin{eqnarray*}
D([e_1, e_2]) & = & D(e_3)=\beta e_n.
\end{eqnarray*}
Now using \eqref{basicfili} we obtain that
\begin{eqnarray*}
[D(e_1), e_2]+ [e_1, D(e_2)] & = & [0, e_2]+[e_1, \alpha
e_{n-1}]=\alpha e_n.
\end{eqnarray*}
Thus $\alpha=\beta.$

Conversely, let $D$ be a linear operator defined by \eqref{filiformder}
with $\alpha=\beta.$ We may assume that $\alpha=\beta=1.$

In order to prove that $D$ is a derivation it is sufficient to show that
$$
D([e_i, e_j])=[D(e_i),e_j]+[e_i,D(e_j)]
$$
for all $1\leq i< j \leq n.$

Case 1. $i+j=3.$ Then $i=1, j=2$ and in this case we can check as
above.

Case 2. $i+j\geq 4.$ Then $j\geq 3,$ and therefore  $[e_i, e_j]\in
\mathcal{L}^{j+1}\subseteq\textrm{span}\{e_4, \cdots, e_n\},$
which implies that $D([e_i, e_j])=0.$

Further $[e_i, D(e_j)]=0,$ because $D(e_j)=e_n\in Z(\mathcal{L})$
or $D(e_j)=0.$

Now let us to show that $[D(e_i), e_j]$ is also zero. We have
\begin{eqnarray*}
[D(e_1), e_j] & = & [0, e_j]=0.
\end{eqnarray*}
Using \eqref{nminus} we obtain that
\begin{eqnarray*}
[D(e_2), e_j] & = & [e_{n-1}, e_j]=0.
\end{eqnarray*}
Finally
\begin{eqnarray*}
[D(e_i), e_j]=0
\end{eqnarray*}
for $i\geq 3,$ because in this case $D(e_i)=e_n$ or $D(e_i)=0.$
So,
\begin{eqnarray*}
D([e_i,  e_j])=0=[D(e_i), e_j]+[e_i, D(e_j)].
\end{eqnarray*}
The proof is complete.
\end{proof}

\begin{remark} It is easy to see that Lemma~\ref{filider} is also
true for the three-dimensional Heisenberg algebra.
\end{remark}

Let $\mathcal{L}$ be a $n$-dimensional filiform Lie algebra with
$n\geq 3.$ Consider  the linear operator $\Delta$ defined by
\eqref{filiformder} with $\alpha=2, \beta=1.$

\begin{lemma}\label{locfili}
The  linear operator $\Delta$ is a local derivation which is not a
derivation.
\end{lemma}

\begin{proof} By Lemma~\ref{filider}, $\Delta$ is not a
derivation.

Let us show that $\Delta$ is a local derivation. Denote by $D_1$
the derivation defined by \eqref{filiformder} with
$\alpha=\beta=1.$ Let $D_2$ be a linear operator on $\mathcal{L}$
defined by
\begin{equation*}
D_2\left(\sum\limits_{k=1}^n x_ke_k\right)=x_2 e_n.
\end{equation*}
Since $D_2|_{\mathcal{[\mathcal{L},\mathcal{L}]}}\equiv 0$ and
$D_2(\mathcal{L})\subseteq Z(\mathcal{L}),$ it follows that
$$
D_2([x,y])=0=[D_2(x),y]+[x,D_2(y)]
$$
for all $x, y\in \mathcal{L}.$ So, $D_2$ is a derivation.

Finally, for any $x=\sum\limits_{k=1}^n x_ke_k$ we find a
derivation $D$ such that $\Delta(x)=D(x).$

Case 1. $x_2=0.$ Then
\begin{eqnarray*}
\Delta(x) & = &  x_3 e_n =  D_1(x).
\end{eqnarray*}

Case 2. $x_2\neq 0.$ Set
$$
D=2D_1+tD_2,
$$
where $t=-\frac{\textstyle x_3}{\textstyle x_2}.$ Then
\begin{eqnarray*}
D(x)  & = &  2D_1(x)+t D_2(x) =  2(x_2e_{n-1}+x_3e_n)+tx_2e_n =
\\
& = & 2 x_2e_{n-1} + (2x_3+tx_2)e_n=2 x_2e_{n-1} +
x_3e_n=\Delta(x).
\end{eqnarray*}
The proof is complete.
\end{proof}

\end{document}